\documentclass[11pt]{amsart}
\usepackage{amsmath}
\newtheorem{theorem}{Theorem}[section]
\newtheorem{definition}[theorem]{Definition}
\newtheorem{lemma}[theorem]{Lemma}
\newtheorem{corollary}[theorem]{Corollary}
\newtheorem{remark}[theorem]{Remark}

\newtheorem{proposition}[theorem]{Proposition}

\newtheorem{problem}[theorem]{Problem}
\makeatletter
\newtheorem*{rep@theorem}{\rep@title}
\newcommand{\newreptheorem}[2]{%
\newenvironment{rep#1}[1]{%
 \def\rep@title{#2 \ref{##1}}%
 \begin{rep@theorem}}%
 {\end{rep@theorem}}}
\makeatother
\newreptheorem{theorem}{Theorem}
\newreptheorem{corollary}{Corollary}

\usepackage{thmtools}

\declaretheoremstyle[headfont=\normalfont]{normalhead}

\newcommand{\R}{\mathbb{R}}

\DeclareMathOperator{\vol}{vol}
\DeclareMathOperator{\e}{\epsilon}

\begin{document}


\title[norm inequalities]{A Busemann-Petty type problem for Dual Radon transforms}

\author[M.~Roysdon]{Michael Roysdon}

\address[M.~Roysdon]{Department of Mathematics, Applied Mathematics, and Statistics, Case Western Reserve, Cleveland, OH 44106, USA.}
\email{{\tt mar327@case.edu}}

\thanks{This research is supported in part by an AMS-Simons Travel Grant}

\begin{abstract} 

Inspired by resolution of the Busemann-Petty problem (1956), we consider the following comparison problem for dual Radon transforms: Given a pair of continuous functions defined on the affine Grassmannian whose dual Radon transforms satisfy a pointwise inequality, can their $L^p$ norms be compared in a meaningful way? We characterize the solution to this problem for each $p \geq 1$, and as a consequence of our investigation, we prove reverse $L^p$-$L^q$-estimates for dual Radon transforms. In particular, we reverse an inequality of Solmon (1979).

\end{abstract}
 
\maketitle

\section{Introduction}
In $1956$ Busemann and Petty in \cite{BP} published a list of ten questions in the geometry of convex bodies, only one of which has been fully resolved. The Busemann-Petty problem (\cite[Problem~1]{BP}) asks the following seemingly simple geometric question: given a pair of origin-symmetric convex bodies (compact, convex subsets with non-empty interior) $K, L \subset \R^n$ such that, for each direction $\theta \in S^{n-1}$,
\begin{equation}\label{e:BPassumption}
\vol_{n-1}(K \cap \theta^\perp) \leq \vol_{n-1}(L \cap \theta^\perp),
\end{equation}
does it necessarily follow that $\vol_n(K) \leq \vol_n(L)$?  Here $\vol_m(\cdot)$ denotes the usual volume in the appropriate dimension, and for $\theta \in S^{n-1}$, $\theta^\perp$ denotes the hyperplane in $\R^n$ orthogonal to $\theta$. The answer to the Busemann-Petty problem was resolved at the end of the 1990s. The answer is affirmative when $n \leq 4$ and negative when $n \ge 5$. For a historical recount and the solution see \cite{Gardner,GKS,Kold1}.

The solution to the Busemann-Petty problem appearing in \cite{GKS} draws on the papers \cite{Lutwak} (due to Lutwak) and \cite{Kold5} (due to Koldobsky). In \cite{Lutwak} the notion of an intersection body of a convex body was introduced, which gave the first affirmative answer to the Busemann-Petty problem in all dimensions for a class of convex bodies. A more general class of convex bodies, called intersection bodies, was examined in \cite{GLW}, and where the authors verified that class of intersection bodies also gives an affirmative answer to the Busemann-Petty problem. 



It was shown in \cite{Lutwak} that the Busemann-Petty problem is affirmative in $\R^n$ if and only if every origin-symmetric convex body is an intersection body. Koldobsky showed in \cite{Kold5} that an origin-symmetric convex body $K \subset \R^n$ is an intersection body if and only if $\|\cdot\|_K^{-1}$ represents a positive definite distribution on $\R^n$.  Here $\|\cdot\|_K$ denotes the norm on $\R^n$ with unit ball $K$. 


The spherical Radon transform of a continuous function $f$ defined on the sphere $S^{n-1}$ is given by 
\[
Rf(\theta) = \int_{S^{n-1} \cap \theta^\perp} f(\xi)d\xi, \quad \theta \in S^{n-1}. 
\]
The next comparison problem for the spherical Radon transform was considered in \cite{KRZ} (see also \cite{Kold25}):

\begin{problem} \label{p:sphericalprolem} Let $p \geq1$. Given a pair of positive, even and continuous functions $f,g$ defined on $S^{n-1}$ such that 
\begin{equation}\label{e:sphericalsomparisonassumpt}
Rf(\theta) \leq Rg(\theta) \quad \text{ for all } \theta \in S^{n-1},
\end{equation}
does it necessarily follow that $\|f\|_{L^p(S^{n-1})} \leq \|g\|_{L^p(S^{n-1})}$?
\end{problem}

It can be shown by integrating the condition \eqref{e:BPassumption} in polar coordinates that Problem~\ref{p:sphericalprolem} contains the Busemann-Petty problem as a special case.  The problem was classified in terms of positive definite distributions in the following sense: the comparison problem is affirmative if and only if every even, continuous, and non-negative function $f \colon S^{n-1} \to (0,\infty)$ satisfies the condition that $|x|_2^{-1} f^{p-1} \left(\frac{x}{|x|_2}\right)$ represents a positive definite distribution on $\R^n$ provided $p > 1$. In particular, there are counterexamples when one considers the case of the Busemann-Petty problem.


  Given an integrable function $\varphi \colon \R^n \to \R$, the classical (spatial) Radon transform of $\varphi$ is defined to be 
\[
\mathcal{R}\varphi(t,\theta) = \int_{x\cdot\theta = t} f(x) dx,
\]
which is well defined almost everywhere, while the dual Radon transform of an even and integrable function $g \colon \R \times S^{n-1} \to \R$ is given by
\[
\mathcal{R}^*g(x) = \int_{S^{n-1}}g(\langle x,\theta \rangle, \theta) d\theta.
 \]

For the properties of the Radon transform, its dual and their applications, we refer the reader to \cite{AGM1,Gardner,Groemer,Helgason,Helgason2,Kold1,Markoe,Rubin1,Schneider}.

 In the present paper, we are interested in studying the next full-dimensional version of the Busemann-Petty problem for the dual Radon transform. 

\begin{problem}[Non-homogeneous Busemann-Petty problem] \label{p:NonhomogeneousBPproblem} Let $p \geq1$. Given a pair of even and continuous functions $g,h \colon \R \times S^{n-1} \to [0,\infty)$ vanishing at infinity (i.e. in $C_0$), with $g,h \in L^1(\R \times S^{n-1}) \cap L^p(\R \times S^{n-1})$, such that $\mathcal{R}^*g(x) \leq \mathcal{R}^*h(x)$
holds for all $x \in \R^n$, does it necessarily follow that 
$\|g\|_{L^p(\R \times S^{n-1})} \leq \|h\|_{L^p(\R \times S^{n-1})}?$
\end{problem}

We have the following characterization for Problem~\ref{p:NonhomogeneousBPproblem}

\begin{theorem}\label{t:basiccase} Let $n \geq 2$ and $p \geq 1$.
\begin{itemize}
    \item[(1)] Given any pair of functions $g,h \in C_0(\R \times S^{n-1}) \cap L^1(\R \times S^{n-1})$ that are even, non-negative and such that $R^*g \leq R^*h$, it follows that $\|g\|_{L^1(\R \times S^{n-1})} \leq \|h\|_{L^1(\R \times S^{n-1})}$.
    \item[(b)] Let $p >1$. Suppose that $g,h \in C_0(\R \times S^{n-1}) \cap L^p(\R \times S^{n-1})$ are even, non-negative functions such that $R^*g \leq R^*h$. If $g^{p-1} = \mathcal{R}\mu$ for some finite and even Borel measure $\mu$ defined on $\R^n$, then $\|g\|_{L^p(\R \times S^{n-1})} \leq \|h\|_{L^p(\R \times S^{n-1})}$. 
    \item[(c)]  Let $p >1$. Suppose that $h \in C_c^\infty(\R \times S^{n-1})$ is even, nonnegative and such that:
    \begin{itemize}
        \item $
    \int_{\R}h(t,\theta) t^m dt$ is an homogeneous $m$-th degree polynomial
    for every $\theta \in S^{n-1}$ and $m \geq 1$, and 
    \item $h = \mathcal{R} \varphi$ for some $\varphi \in C_c^\infty(\R^n)$, with $\varphi <0$ somewhere.
    \end{itemize}
    Then, there exists a non-negative function $g \in C_c^\infty(\R \times S^{n-1})$ such that $\mathcal{R}^*g \leq \mathcal{R}^*h$ and yet $\|g\|_{L^p(\R \times S^{n-1})} > \|h\|_{L^p(\R \times S^{n-1})}$.
\end{itemize}
\end{theorem}



The conclusion of the Busemann-Petty problem is negative in most dimensions, so it makes sense to ask if it holds up to some absolute constant (\cite{MP}): Given any pair of origin-symmetric convex bodies $K,L \subset \R^n$ satisfying the condition \eqref{e:BPassumption}, does it follow that $\vol_n(K) \leq C \vol_n(L)$ for some absolute constant $C>0$? This problem is called the isomorphic Busemann-Petty problem.  It was shown in \cite{MP}, that the isomorphic Busemann-Petty problem is equivalent to the slicing problem of Bourgain \cite{Bourgain0,Bourgain2}: Does there exist an absolute constant $C >0$ such that, for any $n \in \mathbb{N}$ and for any origin-symmetric convex body $K$ in $\R^n$,
\begin{equation}\label{e:Bourgain}
\vol_n(K)^{\frac{n-1}{n}} \leq C \max_{\theta \in S^{n-1}} \vol_{n-1}(K \cap \theta)^{\perp}?
\end{equation}
The history of these two problems is vast and has been investigated by numerous authors. For prominant works concerning the slicing problem, see \cite{Bourgain2, Klartag, Chen, KLeh, JLV, Klartag2, Guan, KLeh2}. In particular, Klartag and Lehec \cite{KLeh2} resolved the slicing problem (and by proxy the isomorphic Busemann-Petty problem) in the affirmative by proving that there is an absolute constant $C >0$ such that 
\begin{equation}\label{e:functionalBourgain}
\|f\|_{L^{\frac{n}{n-1}}(S^{n-1})} \leq C \| Rf\|_{L^\infty(S^{n-1})},
\end{equation}
holds, where $f = \|\cdot\|_K^{-n+1}$ for a convex body $K$ containing the origin. Extensions and analogs of the slicing problem to arbitrary functions were studied in \cite{CGL,GKZ,GK,GrK,KK,KL,Kold2,Kold3,Kold4,KPZ}. Estimates akin to \eqref{e:functionalBourgain} were established by Bennett and Tao in \cite{BT} in the case $0 < p \leq 1$. Moreover, in \cite{KRZ} the next reverse form of \eqref{e:functionalBourgain} was a consequence of the investigation of Problem~\ref{p:sphericalprolem}. Let $p > 1$ and $f \colon S^{n-1} \to (0,\infty)$ be even and continuous.  If $|x|^{-1} f\left(\frac{x}{|x|_2}\right)^{p-1}$ represents a positive definite distribution on $\R^n$, then 
\begin{equation}\label{e:KRZslicing}
\|f\|_{L^{p}(S^{n-1})} \leq \frac{|S^{n-1}|^{\frac{1}{p}}}{|S^{n-2}|} \|Rf\|_{L^\infty(S^{n-1})}.
\end{equation}

As it turns out, our investigation has implications for $L^p$-$L^q$- estimates for the dual Radon transform. $L^p$-$L^q$-estimates for Radon transforms were first initiated by Solmon in \cite{Solmon2,Solmon1}, and expanded upon by works of Oberlin and Stein \cite{OS}, Oberlin \cite{Oberlin}, Calder\'on \cite{Calderon}, Strichartz \cite{Strichartz}, Christ \cite{Christ}, Drury \cite{D1,D2,D3,D4,D5} and Rubin \cite{Rubin2,Rubin}.

The next slicing inequality follows from our investigation of Problem~\ref{p:NonhomogeneousBPproblem}. It can be seen as an analogue of the inequalities \eqref{e:functionalBourgain} and \eqref{e:KRZslicing} for  dual Radon transforms of smooth functions, and also serves as a mean-value property for the dual Radon transform. 

\begin{theorem}\label{t:basicslicing} Let $n\geq 2$ and $p > 1$, and let $w \in L^1(\R \times S^{n-1}) \cap C_c(\R \times S^{n-1})$ be even and non-negative. Then, for any even function $g \in C^\infty(\R \times S^{n-1})$, which decays faster at infinity than any power of $|\cdot|_2$, the following inequality holds:
\[
\|g\|_{L^p(\R \times S^{n-1},w)} \leq \left\|\frac{\mathcal{R}^*[gw]}{\mathcal{R}^*w}\right\|_{L^\infty(\R \times S^{n-1})} \|w\|_{L^1(\R \times S^{n-1})}^{\frac{1}{p}}
\]
\end{theorem}

The slicing inequality appearing in the above theorem may be viewed as a reverse form of $L^p$-$L^q$ inequality due to Solmon \cite{Solmon1} (see inequality \eqref{e:Solmon} below).

Both Theorem~\ref{t:basiccase} and Theorem~\ref{t:basicslicing} will follow from a more general study of a comparison problem for the $(n-k)$-dimensional dual Radon transform in the next section.

\section{Main Results}

Let $1 \leq k <n$,with $k$ and $n$, positive integers. We denote by $\mathcal{G}_{n-k,n}$ the Grassmannian manifold of $(n-k)$-dimensional affine subspaces of $\R^n$ and by $G_{n-k,n}$ the usual Grassmannian manifolds of $(n-k)$-dimensional linear subspaces of $\R^n$ endowed with Haar measure $\nu_{n-k}$ normalized so that $\nu_{n-k}(G_{n-k,n}) = \frac{|S^{k-k}||S^{n-k-1}|}{|S^{n-1}|}$, with the absolute value denoting the volume in the appropriate sense.

The $(n-k)$-dimensional Radon transform of a function $f \in L^1(\R^n)$ is defined as
\[
\mathcal{R}_{n-k}f(H,z) = \int_{H}f(y+z)dy \quad \text{for } (H,z) \in \mathcal{G}_{n-k,n},
\]

and the dual $(n-k)$-dimensional Radon transform of a function $g \in L^1(\mathcal{G}_{n-k,n})$ is defined as
\[
\mathcal{R}^*_{n-k}g(x) = \int_{G_{n-k,n}}g\left(H,P_{H^\perp}(x)\right)dx,
\]
where $P_{H^\perp}(x)$ is the orthogonal projection of $x$ onto $H^\perp$.  When we are dealing with the classical Radon transform corresponding to $k =1$, we have the relations
$
\mathcal{R}_{n-1} = \mathcal{R} \quad \text{ and } \quad \mathcal{R}_{n-1}^* = \frac{1}{2}\mathcal{R}^*. 
$

We will work on the following generalization of Problem~\ref{p:NonhomogeneousBPproblem}. 

\begin{problem}[Comparison problem for the lower dimensional dual Radon transform] \label{p:cpldRt} Let $p \geq 1$ and $f,g \colon \mathcal{G}_{n-k,n} \to [0,\infty)$ be a pair of continuous functions vanishing at infinity such that $f,g \in L^p(\mathcal{G}_{n-k,n})$ and
\begin{equation}\label{e:cpldRtassumption}
\mathcal{R}_{n-k}^*f(x) \leq \mathcal{R}_{n-k}^*g(x) \quad \text{ for all } x \in \R^n.
\end{equation}
Does it necessarily follow that $\|f\|_{L^p(\mathcal{G}_{n-k,n})} \leq \|g\|_{L^p(\mathcal{G}_{n-k,n})}$?
\end{problem}

As we shall see (Proposition~\ref{prop:p=1}) the answer is always yes when $p =1$. However, at the moment $p >1$, the situation becomes more delicate and requires careful investigation. The general idea is to exploit the ideas similar to those in \cite{Gardner,GKS,Lutwak,Kold1,KRZ} by introducing a class of admissible functions, $\mathcal{A}_p^k$, and to show that members of this set of functions classify the solution to Problem~\ref{p:cpldRt} (see below for the associated definitions).  We would also like to remark that different generalizations of the Busemann-Petty problem have been considered in \cite{BZ,KoldGAFA,Milman,Milman2,RZ,Zhang}; however, it doesn't appear that Problem~\ref{p:cpldRt} has any connection to these generalzations except for $k =1$.

Denote by $M^+(\R^n)$ the class of non-negative Radon measure on $\R^n$ having finite total variation.  Denote by $C_b(\mathcal{G}_{n-k,n})$ the space of bounded continuous functions, and by $C_0(\mathcal{G}_{n-k,n})$ those continuous functions that vanish at infinity. Following Helgason, \cite[pg.35]{Helgason2}, we say that a function $g \in \mathcal{C}_c^\infty(\mathcal{G}_{n-k,n})$ satisfies property $(H)$ if, for each $j \in \mathbb{N}$, and $H \in G_{n-k,n}$, the function
$
\mathcal{P}_{H,j}(z) = \int_{H^\perp}g(H + w) (w \cdot z)^j dw, z \in H^\perp$
is the restriction of a homogeneous polynomial of degree $j$ on $\R^n$. We denote this space of functions by $C_{H}^\infty(\mathcal{G}_{n-k,n})$. It is known that (\cite[Ch.~1, Theorem~6.3]{Helgason2}) the $(n-k)$-dimensional Radon transform is a bijection from $C_c(\R^n)$ onto $C_H^\infty(\mathcal{G}_{n-k,n})$. 

\begin{definition} Let $0 < k <n$ be an integer, and let $p >1 $. We define the class of $(p,k)$-admissible functions by 
\[
\mathcal{A}_p^{k} := \left\{h \in C_0(\mathcal{G}_{n-k,n}) \colon h \geq 0 \text{ and } h^{p-1} = \mathcal{R}_{n-k} \mu \text{ for some } \mu \in \mathcal{M}^+(\R^n) \right\}. 
\]
We also have the following important subclass of $\mathcal{I}_p^k$:
\begin{align*}
\mathcal{A}_{p,\infty}^k:= &\left\{h \in C_c^\infty(\mathcal{G}_{n-k,n}) \colon h \geq 0, h^{p-1} \in C_H^\infty(\mathcal{G}_{n-k,n}), \right.\\
&\left. \text{ and } h^{p-1} = \mathcal{R}_{n-k} \varphi  \text{ for some } \varphi \in C_c^\infty(\R^n), \varphi \geq 0 \right\}.
\end{align*}
\end{definition}

We are now in a position to state our first result. 

\begin{theorem} Let $n \geq 2$, $0< k <n$ be an integer, and $p >1$.
\begin{itemize}
    \item[(a)] Suppose that $g,h \in C_0(\mathcal{G}_{n-k,n})$ are such that $\mathcal{R}_{n-k}^*g \leq \mathcal{R}_{n-k}^*h$. If $g \in \mathcal{A}_p^k$, then $\|g\|_{L^p(\mathcal{G}_{n-k,n})} \leq \|h\|_{L^p(\mathcal{G}_{n-k,n})}$. 

    \item[(b)] Let $h \geq 0$ belong to $C_H^\infty(\mathcal{G}_{n-k,n})\setminus \mathcal{A}_{p,\infty}^k$. Then there exists some non-negative function $g \in C_c^\infty(\mathcal{G}_{n-k,n})$ such that $\mathcal{R}_{n-k}^*g  \leq\mathcal{R}_{n-k}^*h $, and yet $\|g\|_{L^p(\mathcal{G}_{n-k,n})} > \|h\|_{L^p(\mathcal{G}_{n-k,n})}$. 
\end{itemize}
    
\end{theorem}

The above theorem is an immediate consequence of Theorem~\ref{t:affirmative} and Theorem~\ref{t:counterexample} below.


One of the prominent results appearing in \cite{Solmon1} is the following mixed norm estimate for the dual $(n-k)$-dimensional Radon transform: Given a measurable function $g \colon \mathcal{G}_{n-k,n} \to \R$ and $p \geq \frac{2k}{n}$, $p > 1$, the inequality
\begin{equation}\label{e:Solmon} 
\|\mathcal{R}_{n-k}^*g\|_{L^q(\R^n)} \leq c \left(\int_{G_{n-k,n}} \|g(H,\cdot)\|_{L^p(H^\perp)}^2 d\nu_{n-k}(H) \right)^{\frac{1}{2}}, \quad  q = \frac{pn}{k}
\end{equation}
holds, where $c>0$ is some constant. Similar estimates have been established by Rubin in \cite{Rubin2}. 

 As a result of our investigation, we establish several reverse forms of \eqref{e:Solmon}. 

\begin{theorem} \label{theorem:slicingaverage} Let $0 < k < n$ be an integer, $p >1$ and $w \in C_0(\mathcal{G}_{n-k,n}) \cap L^1(\mathcal{G}_{n-k,n})$. Let $g\in C_0(\mathcal{G}_{n-k,n}) \cap L^p(\mathcal{G}_{n-k,n},w)$ be such that $\beta e^{-\alpha |z|_2^2} \leq g(H,z) \leq \gamma e^{-\alpha |z|_2^2}$ holds for every $ H \in G_{n-k,n}$ and $z \in H^\perp$, where $\alpha,\beta,\gamma >0$. 
Then
\[
\|g\|_{L^p(\mathcal{G}_{n-k,n},w)} \leq \left(\frac{\gamma}{\beta}\right)^{p-1} \left\|\frac{\mathcal{R}_{n-k}^*[gw]}{\mathcal{R}_{n-k}^*w}\right\|_{L^\infty(\R^n)}\|w\|_{L^1(\mathcal{G}_{n-k,n})}^{\frac{1}{p}}.
\]
\end{theorem}

For $s \geq 0$ consider the Sobolev space
\[
W^{s}(\R^n) =\left \{\varphi \in L^2(\R^n) \colon |\cdot|_2^{s} \widehat{\varphi}(\cdot) \in L^2(\R^n)\right\};
\]
analogously, we define the Sobolev $W^{s}(\mathcal{G}_{n-k,n})$.  Consider the subspace $W_H^s(\mathcal{G}_{n-k,n})$ of all measurable functions $g \colon \mathcal{G}_{n-k,n} \to \R$ satisfying the following conditions: there is a compact convex set in $\R^n$ such that $g$ vanishes whenever it misses $K$, and for every nonnegative integer $m$, there is a homogeneous polynomial $P_m$ of degree $m$ on $\R^n$ such that the restriction of $P_m\mid_{H^\perp}$ to the $k$-plane $H^\perp$ satisfies
$
P_m\mid_{H^\perp}(z) = \int_{H^\perp}(z \cdot y)^m g(H,z) dz. 
$
For more precise definitions, see Section \ref{sec:background}.

\begin{theorem}\label{t:slicing} Let $0 < k < n$ be an integer and $s \geq 0$, and $p >1$. If $g^{p-1}$ belongs to the Sobolev space $W_H^{s + \frac{n-k}{2}}(\mathcal{G}_{n-k,n}) $ is non-negative and continuous, then the following estimate holds: 
\[
\frac{\|g\|_{L^p(\mathcal{G}_{n-k,n})}}{\|g^{p-1}\|_{L^p(\mathcal{G}_{n-k,n})}} \leq \left\|\frac{\mathcal{R}_{n-k}^* g}{\mathcal{R}_{n-k}^* g^{p-1}}\right\|_{L^\infty(\R^n)}
\]
\end{theorem}

Another interesting consequence of Theorem~\ref{t:generalslicing} is the following mean-value property for the dual $(n-k)$-dimensional Radon transform. 

\begin{theorem}\label{t:meanvalue} Let $0 < k < n$ be an integer, $p > 1$, $s \geq 0$, and $w \colon \mathcal{G}_{n-k,n} \to [0,\infty)$ be a compactly supported continuous function.  If $g^{p-1} \in W_H^{s+ \frac{n-k}{2}}(\mathcal{G}_{n-k,n})$ is non-negative and continuous, then 
\[
\frac{\|g\|_{L^p(\mathcal{G}_{n-k,n},w)}}{\|w\|_{L^1(\mathcal{G}_{n-k,n})}^{\frac{1}{p}}} \leq \left\| \frac{\mathcal{R}_{n-k}^*gw}{\mathcal{R}_{n-k}^*w}\right\|_{L^\infty(\R^n)}.
\]   
\end{theorem}

The paper is organized as follows. In Section~\ref{sec:background}, we introduce the necessary background for defining the $(n-k)$-dimensional Radon transform of a measure and establish mapping properties for the dual Radon transform (Theorem~\ref{t:mappingproperty}). In Section~\ref{sec:solution}, we provide a solution to the Problem~\ref{p:cpldRt}. Finally, in Section~\ref{sec:slicing}, we provide the proofs of Theorem~\ref{theorem:slicingaverage}, Theorem~\ref{t:slicing}, and Theorem~\ref{t:meanvalue} as a byproduct of a more general result, Theorem~\ref{t:generalslicing}.





    

\section{The $(n-k)$-dimensional Radon transform of a measure: definition and mapping properties} \label{sec:background}

In this section, we outline sufficient background material to define the $(n-k)$-dimensional Radon transform of a measure. To do this, we extend the construction in \cite[Section~2]{BL} to the case where $k >1$.

We will work in the $n$-dimensional Euclidean space $\R^n$ equipped with its usual inner product structure $x \cdot y$ and induced normed $|x|_2 = \sqrt{x \cdot x}$. We denote the Lebesgue measure of a measurable subset $A$ of $\R^n$ of appropriate dimension by $\vol_m(A)$. The $n$-dimensional Euclidean unit ball shall be denoted by $B_2^n$, and its boundary, the unit sphere, by $S^{n-1}$.


Given a measure metric space $(X,d,\mu)$ and $p \neq 0$, we say that a real-valued function $h\colon X \to \R$ belongs to $L^p(X)$ if 
\[
\int_X |h(x)|^p d\mu(x) <\infty.
\]
We define the $L^p(X)$-norm of a function $h \colon X \to \R$ to be 
\[
\|h\|_{L^p(X)} = \left(\int_{X}|h(x)|^pd\mu(x) \right)^{\frac{1}{p}}.
\]

The Fourier transform of $f \in L^1(\R^n)$ is given by 
\[
\widehat{f}(x) = \int_{\R^n} f(y) e^{-i x \cdot y} dy.
\]

Define $C_c^{\infty}(\R^n)$ to be the space of compactly supported real-valued functions endowed with the supremum norm, and let  $C_0(\R^n)$ be its closure in the supremum norm, that is, the space of continuous real-valued functions on $\R^n$ vanishing at infinity.  

Recall that a Radon measure on $\R^n$ is a regular Borel measure that is finite on all compact sets. A signed Radon measure is a Borel measure whose positive and negative variations are themselves Radon measures. According to the Riesz representation theorem, \cite[Theorem~7.17]{Folland}, the dual space of $C_0(\R^n)$ is the space of signed Radon measures, $M(\R^n)$, endowed with the total variation norm, i.e. given $\mu \in M(\R^n)$, $\|\mu\|_{(C_0(\R^n))^*} = |\mu|(\R^n)$, having finite total variation. The action of a member $\mu \in M(\R^n)$ on a test function $\varphi \in C_0(\R^n)$ will be denoted by the pairing $\langle \mu, \varphi \rangle$.   In fact, even more can be gleaned from the Riesz representation theorem; it guarantees that to any continuous linear functional $\Phi$ acting on $C_0(\R^n)$ there corresponds a unique signed Radon measure $\mu$ such that $\langle \Phi, \varphi \rangle =\int_{\R^n} \varphi d\mu$ for all $\varphi \in C_0(\R^n)$. Conversely, if the set function $\mu(A)$ is given, then the previous integral defines a continuous linear functional on $C_0(\R^n)$.  Furthermore, we remark that the space $L^1(\R^n)$ is a subspace of $M(\R^n)$ via the identification $\varphi \mapsto \int f(x) \varphi(x) dx$ for $f \in L^1(\R^n)$. 

Finally, we note that any member $\mu \in M(\R^n)$ can be unique extended as a linear functional to the space of continuous and bounded functions defined on $\R^n$, which we denote by $C_b(\R^n)$, as follows: take a compactly supported continuous function $\chi$ that is equal to $1$ near the origin, and then we define 
\begin{equation} \label{e:functionalextension}
\langle \mu, \varphi \rangle := \lim_{r \to \infty}\left\langle \mu, \chi\left(\frac{\cdot}{r} \right)\varphi(\cdot) \right\rangle, \quad \varphi \in C_b(\R^n).
\end{equation}

Extending $\mu \in M(\R^n)$ to $C_b(\R^n)$ in the sense of \eqref{e:functionalextension}, we can define its Fourier transform in the following way
\begin{equation}\label{e:FourierMeasure}
\widehat{\mu}(\xi) :=  \langle \mu, \xi \mapsto e^{-i x \cdot \xi} \rangle = \int_{\R^n} e^{-i x \cdot \xi} d\mu(x), \quad \xi \in \R^n. 
\end{equation}
It is also apparent that $\widehat{\mu}$ is a bounded and continuous function, and that $\|\widehat{\mu}\|_{\infty} \leq |\mu|(\R^n)$. 
For more details see \cite[Chapter~8]{Folland}. 

Let $1 \leq k <n$ be  an integer. We denote by $\mathcal{G}_{n-k,n}$ the Grassmannian manifold of $(n-k)$-dimensional affine subspace of $\R^n$ and by $G_{n-k,n}$ the usual Grassmannian manifolds of $(n-k)$-dimensional linear subspaces of $\R^n$ endowed with Haar measure $\nu_{n-k}$, and let $|G_{n-k,n}| = \nu_{n-k,n}(G_{n-k,n})$ be its volume.  
For $p > 0$, given a function $w \colon \mathcal{G}_{n-k,n} \to [0,\infty)$, we define the space $L^p(\mathcal{G}_{n-k,n},w)$ to be the space of (equivalence classes of) functions $g \colon \mathcal{G}_{n-k,n} \to \R$ such that 
\[
 \int_{G_{n-k,n}}\int_{H^\perp} |g(H,z)|^p w(H,z) dz \nu_{n-k}(H)\ < \infty
\]
with the associated norm
\[
\|g\|_{L^p(\mathcal{G}_{n-k,n},w)} =\left( \int_{G_{n-k,n}}\int_{H^\perp} |g(H,z)|^p w(H,z) dz \nu_{n-k}(H)\right)^{\frac{1}{p}}
\]
for $g \colon \mathcal{G}_{n-k,n} \to \R$. When $w \equiv 1$, we simply write $L^p(\mathcal{G}_{n-k,n})$ in place of $L^p(\mathcal{G}_{n-k,n},1)$.

 The $(n-k)$-dimensional Radon transform of a function $f \in L^1(\R^n)$ is defined almost everywhere by
\[
\mathcal{R}_{n-k}f(H,z) = \int_{H}f(y+z) dy, 
\]
where $(H,z) \in  \mathcal{G}_{n-k,n}$ and $dz$ is the Lebesgue measure on the space $H$.  In the case where $k=1$, we recover a scalar multiple of a classical Radon transform defined in the introduction.  According to \cite[Corollary~3.25]{Markoe}
\[
\|\mathcal{R}_{n-k} f\|_{L^1(\mathcal{G}_{n-k,n})} \leq |G_{n-k,n}| \|f\|_{L^1(\R^n)}, \quad \text{ for all } f \in L^1(\R^n), 
\]
which implies that $\mathcal{R}_{n-k}f \in L^1(\mathcal{G}_{n-k})$ whenever $f \in L^1(\R^n)$. 

We define the Fourier transform on $\mathcal{G}_{n-k,n}$ as follows. If $(H,z) \in \mathcal{G}_{n-k,n}$ and $g$ is in $L^1(\mathcal{G}_{n-k,n})$ function in the second variable, then 
\[
\mathcal{F}_kg(H,y) = \int_{H^{\perp}}e^{-i  y\cdot z } g(H,z) dz.
\]
Given a function $f \colon \R^n  \to \R$ and a set $A \subset \R^n$, we denote by $f \mid_A$ the restriction of $f$ to the set $A$.  The Fourier-Slice Theorem \cite[Theorem~3.27]{Markoe} connects the $(n-k)$-dimensional Radon transform to the Fourier transform in the sense that if $f \in L^1(\R^n)$ and $0 < k < n$, then, for each $H \in G_{n-k,n}$,
\begin{equation}\label{e:FSThm}
\mathcal{F}_{k}\mathcal{R}_{n-k} f(H,z) = \widehat{f}\mid_{H^\perp}(z), \quad \text{for } z \in H^{\perp}. 
\end{equation}

An immediate consequence of \eqref{e:FSThm} is that the $(n-k)$-dimensional Radon transform is an injection from $L^1(\R^n)$ into $L^1(\mathcal{G}_{n-k,n})$. In particular, if $f,h \in L^1(\R^n)$ have the same $(n-k)$-dimensional Radon transform, then $f=h$.

Let $C_0(\mathcal{G}_{n-k,n})$ be the space of real-valued continuous functions defined on $\mathcal{G}_{n-k,n}$ vanishing at infinity, endowed with the supremum norm. Given $g \in L^1(\mathcal{G}_{n-k,n})$ and $\psi \in C_0(\mathcal{G}_{n-k,n})$, we consider the pairing 
\[
\langle g, \psi \rangle_k := \int_{G_{n-k,n}} \int_{H^\perp}g(H,z) \psi(H,z) dz d\nu_{n-k}(H). 
\]
If $\Psi \in (C_0(\mathcal{G}_{n-k,n}))^*$, the dual space of $C_0(\mathcal{G}_{n-k,n})$, then we write $\langle \Psi, \psi \rangle_k$ to denote the action of $\Psi$ on $\psi \in C_0(\mathcal{G}_{n-k,n})$. In this way, $L^1(\mathcal{G}_{n-k,n})$ is identified with the subspace of $(C_0(\mathcal{G}_{n-k,n}))^*$ given by 
\[
\psi \in C_0(\mathcal{G}_{n-k,n}) \mapsto \int_{G_{n-k,n}}\int_{H^\perp} g(H,z) \psi(H,z)dz d\nu_{n-k}(H), \quad g \in L^1(\mathcal{G}_{n-k,n}).
\]

The dual $(n-k)$-dimensional Radon transform of a function $\psi \in C_0(\mathcal{G}_{n-k,n})$ is defined as 
\[
\mathcal{R}^*_{n-k}\psi(x) = \int_{G_{n-k,n}}\psi\left(H,P_{H^\perp}(x)\right)d\nu_{n-k}(H).
\]  

Denote by $C_b(\mathcal{G}_{n-k,n})$ the space of real-valued bounded and continuous functions on $ \mathcal{G}_{n-k,n}$. It is clear that $\mathcal{R}_{n-k}^*$ maps the spaces $ C_0(\mathcal{G}_{n-k,n})$ and $ C_b(\mathcal{G}_{n-k,n})$ into $ C_b(\R^n)$; however, as the next result shows, much more can be said. Let  $C_c(\mathcal{G}_{n-k,n}) \subset  C_0(\mathcal{G}_{n-k,n})$ the dense subspace of compactly supported functions. 

\begin{theorem}\label{t:mappingproperty} Let $0 < k <n$ be an integer. The following hold: 
\begin{itemize}
    \item[(a)]  Suppose that $\psi \in C_c(\mathcal{G}_{n-k,n})$. Then $R_{n-k}^*\psi(x) = O(|x|_2^{-k})$ as $|x|_2 \to \infty$. 
    \item[(b)] The dual Radon transform $R_{n-k}^*$ maps $C_0(\mathcal{G}_{n-k,n})$ continuously into $C_0(\R^n)$ and satisfies the estimate
    \[
    \sup |\mathcal{R}_{n-k}^* \psi| \leq \sup |\psi|, \quad \text{ for all } \psi \in C_0(G_{n-k,n}). 
    \]
\end{itemize}
\end{theorem}

Before proceeding to the proof, we need to introduce bi-spherical coordinates following very closely \cite{Markoe}. The bi-spherical coordinate system is defined by the map $\Psi: S^{n-k-1} \oplus S^{k-1} \times \left[0, \frac{\pi}{2} \right] \to S^{n-1}$ via 
$
\Psi(\theta,\omega,\beta) = \cos(\beta) \theta + \sin(\beta) \omega. 
$
Here $\oplus$ denotes the direct sum. 
We have the change of variables formula: Given a function $f \colon S^{n-1} \to \R$, 
\begin{equation}\label{e:bisphereicalcoordinates}
\begin{split}
&\int_{S^{n-1}}f(\xi) d \xi  = \int_{S^{n-k-1} \oplus S^{k-1} \times \left[0, \frac{\pi}{2} \right]} f(\cos(\beta)\theta + \sin(\beta)\omega)\\ 
&\times \sin^{k-1}(\beta)\cos^{n-k-1}(\beta)d\theta d\omega d\beta. 
\end{split}
\end{equation}

\begin{proof}[Proof of Theorem~\ref{t:mappingproperty}]  Given $r > 0$ and $x \in \R^n \setminus \{0\}$ such that $|x|_2 \geq r$, consider the set $\{H \in G_{n-k,n} \colon |P_{H^\perp}(x)|_2 \leq r \}$. Then, using rotation invariance of the Haar measure, we can write  
\begin{align*}
\int_{\left\{H \in G_{n-k,n} \colon \left|P_{H^\perp}(x)\right|_2 \leq r \right\}} d\nu_{n-k}(H) &= \int_{SO(n)} \chi_{\{A \in SO(n) \colon d(x,A\R^k) \leq r\}}(A)dA\\
&= \int_{SO(n)}\chi_{\{A \in SO(n) \colon d(A^{-1}x,\R^k) \leq r\}}(A)dA\\
&=\frac{1}{|S^{n-1}|}\int_{S^{n-1}}\chi_{\{\theta\in S^{n-1} \colon d(|x|_2\theta ,\R^k) \leq r\}}(\theta)d\theta\\
&=\frac{1}{|S^{n-1}|}\int_{S^{n-1}}\chi_{\left\{\theta\in S^{n-1} \colon \left|P_{\R^k}(|x|_2\theta)\right|_2 \leq \frac{r}{|x|_2}\right\}}(\theta)d\theta.
\end{align*} 
Set 
\[
\theta = \cos(\omega) a + \sin(\omega) b, \quad a \in S^{n-k-1} \subset \R^{n-k-1}, b \in S^{k-1} \subset \R^k, \omega \in \left[0, \frac{\pi}{2} \right].
\]
Since $a$ is orthogonal to $\R^{k}$,
$P_{\R^k}(\theta)=\sin(\omega)b.$
Moreover, as $b$ is a unit vector, for $0 \leq \omega \leq \frac{\pi}{2}$, we see that 
\[
|P_{\R^k}(\theta)|_2=|\sin(\omega)b|_2 = |b \sin(\omega)|_2 = \sin(\omega). 
\]
Applying bi-spherical coordinates \eqref{e:bisphereicalcoordinates} we obtain
\begin{align*}
&\int_{\left\{H \in G_{n-k,n} \colon \left|P_{H^\perp}(x)\right|_2 \leq r \right\}} d\nu_{n-k}(H) =\frac{1}{|S^{n-1}|}\int_{S^{n-1}}\chi_{\left\{\theta\in S^{n-1} \colon \left|P_{\R^k}(|x|_2\theta)\right|_2 \leq \frac{r}{|x|_2}\right\}}(\theta)d\theta \\
&= \frac{1}{|S^{n-1}|}\int_{S^{k-1}}\int_{S^{n-k-1}}\int_0^{\frac{\pi}{2}} \chi_{\left\{\cos(\omega)a + \sin(\omega) b \colon \sin(\omega) \leq \frac{r}{|x|_2}\right\}}(\cos(\omega)a + \sin(\omega) b)\\
&\times \sin^{k-1}(\omega)\cos^{n-k-1}(\omega)d\omega dadb\\
&=\frac{|S^{k-1}||S^{n-k-1}|}{|S^{n-1}|} \int_0^{\sin^{-1}(r|x|_2^{-1})}\sin(\omega)^{k-1}\cos^{n-k-1}(\omega)d\omega\\
&\leq \left(\frac{\pi}{2}\right)^k\frac{|S^{k-1}||S^{n-k-1}|}{|S^{n-1}|} r^k |x|_2^{-k} ,
\end{align*}
where we have employed the inequalities 
\[
 \sin(\omega) \leq \omega, \omega \in \left[0,\frac{\pi}{2}\right] \text{ and } \sin^{-1}(t) \leq \frac{\pi t}{2}, t \in [0,1].
\]
The previous calculations show that 
\[\int_{\left\{H \in G_{n-k,n} \colon \left|P_{H^\perp}(x)\right|_2 \leq r \right\}} d\nu_{n-k}(H) \leq \left(\frac{\pi}{2}\right)^k\frac{|S^{k-1}||S^{n-k-1}|}{|S^{n-1}|} r^k |x|_2^{-k}.
\] 
Next, we use the estimates 
\[
cos(\omega) \geq 1- \frac{\omega}{\sin^{-1}(t)}, t \in (0,1], \sin(\\\omega) \geq \frac{\omega}{2}, \omega \in \left[0,\frac{\pi}{2}\right], \text{ and } \sin^{-1}(t) \geq t, t \in [0,1]
\]
to get a lower bound: 
\begin{align*}
&\int_{\left\{H \in G_{n-k,n} \colon \left|P_{H^\perp}(x)\right|_2 \leq r \right\}} d\nu_{n-k}(H)\\
&\geq \frac{|S^{k-1}||S^{n-k-1}|}{2^{k-1}|S^{n-1}|} \int_0^{\sin^{-1}(r|x|_2^{-1})}\omega^k\left(1- \frac{\omega}{\sin^{-1}(r|x|_2^{-1})} \right)^{n-k-1} d\omega \\
&= \frac{|S^{k-1}||S^{n-k-1}|}{2^{k-1}|S^{n-1}|} \frac{\Gamma(k)\Gamma(n-k)}{\Gamma(n)} \sin^{-1}(r|x|_2^{-1})^k\\
&\geq \frac{|S^{k-1}||S^{n-k-1}|}{2^{k-1}|S^{n-1}|} \frac{\Gamma(k)\Gamma(n-k)}{\Gamma(n)} r^k |x|_2^{-k},
\end{align*}
where $\Gamma$ denotes the gamma function. 

Thus, we have shown that, given any $r > 0$ and $x \in \R^n \setminus \{0\}$, with $|x|_2 \geq r$, there are constant $m(n,k), M(n,k)>0$, depending only on $n$ and $k$ such that 
 \begin{equation}\label{e:caps}
m(n,k) r^k |x|_2^{-k} \leq \int_{\left\{H \in G_{n-k,n} \colon \left|P_{H^\perp}(x)\right|_2 \leq r \right\}} d\nu_{n-k}(H) \leq M(n,k) r^k|x|_2^{-k}. 
 \end{equation}


Let $\psi \in C_c(\mathcal{G}_{n-k,n})$. If $\psi$ is supported on a compact set $S \subset \mathcal{G}_{n-k,n}$, then there must be an $r_0 > 0$ sufficiently large so that $S \subset G_{n-k,n} \times r_0B_2^k$. Then, according to \eqref{e:caps}, whenever $r > r_0$ and $|x|_2 \geq r$,
\begin{align*}
|\mathcal{R}_{n-k}^*\psi(x)| &\leq  \int_{S}\left|\psi\left(H,P_{H^\perp}(x)\right) \right|d\nu_{n-k,n}(H)\\
&\leq \sup |\psi| \int_{\left\{H \in G_{n-k,n} \colon \left|P_{H^\perp}(x)\right|_2 \leq r \right\}} d\nu_{n-k}(H)\\
&\leq M(n,k) \sup |\psi| r^k|x|_2^{-k},
\end{align*}
which completes the proof of (a). 

Finally, for part (b), suppose that $\psi \in C_0(\mathcal{G}_{n-k,n})$ let $\e >0$, and choose $r >0$ sufficiently large so that $|\psi(H,z)| < \e$ whenever $|z|_2 > r$ for $z \in H^\perp$ for each $H \in G_{n-k,n}$.  By \eqref{e:caps}, we may additionally choose $x \in \R^n$ such that $|x|_2$ is large enough to guarantee that $\nu_{n-k}\left(\left\{H \in G_{n-k,n} \colon \left|P_{H^\perp}(x)\right|_2 \leq r \right\}\right) < \e$. Then, for all such $x \in \R^n$, we have that 
\begin{align*}
|R_{n-k}^*\psi(x)| &\leq \int_{\left\{H \in G_{n-k,n} \colon \left|P_{H^\perp}(x)\right|_2 \leq r \right\}} \left|\psi\left(H,P_{H^\perp}(x)\right)\right|d\nu_{n-k}(H)\\
&+ \int_{\left\{H \in G_{n-k,n} \colon \left|P_{H^\perp}(x)\right|_2 > r \right\}} \left|\psi\left(H,P_{H^\perp}(x)\right)\right|d\nu_{n-k}(H)\\
\leq \e (\sup|\psi| + |G_{n-k,n}|),
\end{align*}
as required.  The final estimate is obvious. 
\end{proof}

The inequality \eqref{e:caps} may be of independent interest, so we isolate it in the next theorem. 

\begin{theorem} Let $r > 0$ and $ x \in \R^n \setminus \{0\}$ be such that $|x|_2 \geq r$. Then, there exist constants $m(n,k),M(n,k) > 0$, depending only on $n$ and $k$ such that 
\[
m(n,k) r^k |x|_2^{-k} \leq \nu_{n-k}\left( \left\{H \in G_{n-k,n} \colon |P_{H^\perp}(x)|_2 \leq r\right\} \right) \leq M(n,k) r^k |x|_2^{-k}. 
\]
    
\end{theorem}

Consider the subset $C_k^\infty(\R^n)$ of $C^\infty(\R^n)$ defined by
\[
C_k^\infty(\R^n) := \{\varphi \in C^\infty(\R^n) \colon \varphi(x) = O(|x|_2^{-k}) \text{ as } |x|_2 \to \infty\}. 
\]

 Gonzalez, \cite[Theorem~5.1]{Gonzalez}, showed that $R_{n-k}^*$ continuously maps $C^{\infty}(\R^n)$ onto the space  $C^\infty(\mathcal{G}_{n-k,n})$. The next theorem is an analogue of \cite[Theorem~1.4(c)]{Hertle} and \cite[Remark~1.6]{Hertle} and may be of independent interest. To our knowledge, it has not appeared in the literature except when $k=1$.  

\begin{theorem}\label{t:mappingpropertiesinfinitelysmooth} Let $0 < k <n$ be an integer. The following hold: 
\begin{itemize}
    \item[(a)] For any $\psi \in C_c^\infty(\mathcal{G}_{n-k,n})$ it holds that $R_{n-k}^*\psi(x) = O(|x|_2^{-k})$ as $|x|_2 \to \infty$. In other words, the $(n-k)$-dimensional dual Radon transform $R_{n-k}^*$ maps $C_c^\infty(\mathcal{G}_{n-k,n})$ continuously into $C_k^\infty(\R^n)$.
    \item[(b)] The dual Radon transform $R_{n-k}^*$ maps $C_0^\infty(\mathcal{G}_{n-k,n})$ continuously into $C_0^\infty(\R^n)$ and satisfies the estimate
    \[
    \sup |\mathcal{R}_{n-k}^* \psi| \leq \sup |\psi|, \quad \text{ for all } \psi \in C_0(G_{n-k,n}). 
    \]
\end{itemize}
    
\end{theorem}

\begin{proof} By \cite[Theorem~5.1]{Gonzalez}, $\mathcal{R}_{n-k}^*(C^\infty(\mathcal{G}_{n-k,n})) = C^\infty(\R^n)$. Repetition of the proof of Theorem~\ref{t:mappingproperty} and replacement of $C_c(\mathcal{G}_{n-k,n})$, $C_0(\mathcal{G}_{n-k,n})$ with the spaces $C_c^\infty(\mathcal{G}_{n-k,n})$, $C_0^\infty(\mathcal{G}_{n-k,n})$, respectively, give the result. 
\end{proof}

\begin{remark} The range of the dual Radon transform, $\mathcal{R}_{n-1}^*$ on the class of Schwartz test functions was studied in \cite{Solmon}, by employing Fourier analytic techniques. The classification of the range of $\mathcal{R}_{n-k}^*$ restricted to the class $C_c^\infty(\mathcal{G}_{n-k,n})$ (including the class $k=1$) appears to be a very difficult (and interesting) open question. 
    
\end{remark}

 Theorem~\ref{t:mappingproperty}(b) implies that $\mathcal{R}_{n-k}^*\psi \in C_0(\R^n)$ whenever $\psi \in C_0(\mathcal{G}_{n-k,n})$. Therefore, we may define the $(n-k)$-dimensional Radon transform of a measure $\mu \in M(\R^n)$ via duality: 
\begin{equation}\label{e:RadonDef}
\langle \mathcal{R} \mu, \psi\rangle_k = \langle \mu, \mathcal{R}_{n-k}^*\psi  \rangle \quad \text{ for all } \psi \in C_0(\mathcal{G}_{n-k,n}). 
\end{equation}
It follows that $\mathcal{R}_{n-k} \mu \in M(\mathcal{G}_{n-k,n})$, the space of signed Radon measures on $\mathcal{G}_{n-k,n}$ with finite total variation, and furthermore that $|R_{n-k}\mu|(\mathcal{G}_{n-k,n}) \leq |\mu|(\R^n)$, for each $\mu \in M(\R^n)$.  

Our next step is to establish an extension of the Fourier-Slice theorem, \eqref{e:FSThm}, to the space $M(\R^n)$, and to do so, we must extend $R_{n-k} \mu$ to act on $C_b(\mathcal{G}_{n-k,n})$. To do this, we have the following lemma.

\begin{lemma} \label{l:helpful}  Let $0 < k < n$ be an integer, and $\mu \in M(\R^n)$. Then, for every $H \in G_{n-k,n}$, the restriction $(\mathcal{R}_{n-k}\mu)_H(\cdot) \in M(H^\perp)$ is well-defined, where $(\mathcal{R}_{n-k}\mu)_H(\cdot) = \mathcal{R}_{n-k}\mu(H,\cdot)$ is the restriction of $\mathcal{R}_{n-k} \mu$ to $H^\perp$. Moreover, for every $H \in G_{n-k,n}$,
\[
\langle (\mathcal{R}_{n-k}\mu)_H(\cdot), \varphi \rangle_{H^\perp} = \langle \mu, \varphi\left(P_{H^\perp}(x)\right) \rangle_{H^\perp}, \quad \varphi  \in C_0(H^\perp),
\]
where $\langle \cdot, \cdot\rangle_{H^\perp}$ is pairing between $C_0(H^\perp)$ and $M(H^\perp)$. 

\end{lemma}

\begin{proof} Fix $H \in G_{n-k,n}$ and $\mu \in M(\R^n)$. For brevity, we set 
\[
\pi(x) := P_{H^\perp}(x),  \quad x \in \R^n. 
\]
and define the push-forward of $\mu$ under $\pi$ by $\pi_* \mu(E) = \mu(\pi^{-1}(B))$ for every Borel set $B \subset H^\perp$. Then, as a set function, it becomes apparent that $R_{n-k}\mu(H,\cdot)$ coincides with $\pi_* \mu$. 
Now we wish to view $\mu$ as a linear functional. Recall that the pullback of $\pi^*$ on test functions is understood via
\[
\pi^* \varphi := \varphi \circ \pi \in C_b(\R^n), \quad \varphi \in C_0(H^\perp). 
\]
Then, we may define the push-foward $\pi_* \mu$, as a linear functional, by 
\[
\langle \pi_*\mu, \varphi \rangle = \langle \mu, \pi^* \varphi \rangle = \langle \mu, \varphi(\pi(x)) \rangle, \quad \varphi \in C_0(H^\perp). 
\]
This completes the proof. 
\end{proof}

Similar to $\widehat{\mu}$, with $\mu \in M(\R^n)$, we can extend $\mathcal{F}_k$ from $L^1(\mathcal{G}_{n-k,n})$ to $M(\mathcal{G}_{n-k,n})$ via
\[
\mathcal{F}_k \nu(\omega) := \langle \nu, z \mapsto e^{-i z \cdot \omega } \rangle_{H^\perp} = \int_{H^\perp} e^{-i \omega \cdot z } d\nu(z), \quad H \in G_{n-k,n}, \omega \in H^\perp. 
\]
It is clear that $\mathcal{F}_k \nu$ is a bounded continuous function satisfying $\sup \mathcal{F}_k \nu \leq |\nu|(H^\perp)$ for every $H \in G_{n-k,n}$. 

We are now in a position to prove the Fourier-Slice Theorem for measures.

\begin{theorem}\label{t:FSMeasures} Let $0 < k <n$ be an integer, and fix $\mu \in M(\R^n)$ and $H \in G_{n-k,n}$. Then 
\[
\mathcal{F}_k [(\mathcal{R}_{n-k}\mu)_H](\omega) = \widehat{\mu} \mid_{H^\perp}(\omega), \quad \omega \in H^\perp. 
\]
Here $\widehat{\mu} \mid_{H^\perp}$ is the restriction of $\widehat{\mu}$  to $H^\perp$. In particular, $\mathcal{R}_{n-k} \colon M(\R^n) \to M(\mathcal{G}_{n-k})$ is an injective bounded linear operator. 
\end{theorem}

\begin{proof} Set $\nu = \mathcal{R}_{n-k}\mu(H,\cdot)$ Applying the final conclusion of Lemma~\ref{l:helpful}, we have 
\begin{align*}
\mathcal{F} \nu(\omega) &=  \langle \nu, \omega \mapsto e^{-i \omega \cdot z} \rangle_{H^\perp}= \left\langle \mu, \xi \mapsto e^{-i \omega \cdot P_{H^\perp}(\xi)} \right\rangle=\widehat{\mu} \mid_{H^\perp}(\omega),
\end{align*}
and the proof of the first part is complete. 

To establish the final part, note that if $\mathcal{R}_{n-k}\mu = 0$, then the Fourier-Slice theorem demands that, for every $H \in G_{n-k,n}$, we have 
\[
\mathcal{F}_k \mathcal{R}_{n-k}\mu(H, \omega) = \widehat{\mu} \mid_{H^\perp}(\omega) = 0
\]
for all $\omega \in H^\perp$. Observing that $\bigcup_{H \in G_{n-k,n}} H^\perp = \R^n$, it follows $\widehat{\mu}(\xi) = 0$ for every $\xi \in \R^n$. Therefore, $\mu = 0$, as required. 
\end{proof}

Combining Theorem~\ref{t:mappingpropertiesinfinitelysmooth} with Theorem~\ref{t:FSMeasures}, we obtain the following useful result.

\begin{corollary}\label{coro:mapping} Let $0 <k <n$ be an integer. The map $\mathcal{R}_{n-k}^*$ is bounded linear operator and a bijection from $C_c^\infty(\mathcal{G}_{n-k,n})$ onto the subspace 
\[
\mathcal{R}_{n-k}^*(C_c^\infty(\mathcal{G}_{n-k,n})) := \{\varphi \in C_k^\infty(\R^n) \colon \varphi = \mathcal{R}_{n-k}^*g \text{ for some } g \in C_c^\infty(\mathcal{G}_{n-k,n})\}
\]
of $C_k^\infty(\R^n)$. 
    
\end{corollary}

To conclude the section, we require some notation. 

Following Helgason, \cite[pg.35]{Helgason2} (see also \cite{Helgason, Markoe}), we say that a function $g \in \mathcal{C}_c^\infty(\mathcal{G}_{n-k,n})$ satisfies property $(H)$ if, for each $j \in \mathbb{N}$, and $H \in G_{n-k,n}$, the function
\[
\mathcal{P}_{H,j}(z) = \int_{H^\perp}g(H + w) (w \cdot z)^j dw, z \in H^\perp
\]
is the restriction of a homogeneous polynomial of degree $j$ on $\R^n$. We denote this space of functions by $C_H^\infty(\mathcal{G}_{n-k,n})$.

The next lemma is due to Helgason, \cite[Ch.~1, Theorem~6.3]{Helgason2}. 

\begin{lemma} \label{l:additionalmapping} 

    
    The $(n-k)$-dimensional Radon transform is a bijection of $C_c^\infty(\R^n)$ onto $C_H^\infty(\mathcal{G}_{n-k,n})$. 

\end{lemma}

Finally, we examine Sobolev functions. Let $s \geq 0$, and consider the Sobolev spaces 
\begin{align*}
W^s(\R^n) &= \{\varphi \in L^2(\R^n) \colon |\cdot|_2^s \widehat{f}(\cdot) \in L^2(\R^n)\},\\
W^s(\mathcal{G}_{n-k,n}) &= \{g \in L^2(\R^n) \colon |\cdot|_2^s \mathcal{F}_k g(H,\cdot) \in L^2(\mathcal{G}_{n-k,n})\}.
\end{align*}
Denote by $W_c^s$ those compactly supported functions in $W^s$. Finally, consider the subset $W_H^s(\mathcal{G}_{n-k,n})$ of all measurable functions $g \colon \mathcal{G}_{n-k,n} \to \R$ satisfying the following conditions: there is a compact convex set in $\R^n$ such that $g$ vanishes whenever it misses $K$, and for every nonnegative integer $m$, there is a homogeneous polynomial $P_m$ of degree $m$ on $\R^n$ such that the restriction of $P_m\mid_{H^\perp}$ to the $k$-plane $H^\perp$ satisfies
$
P_m\mid_{H^\perp}(z) = \int_{H^\perp}(z \cdot y)^m g(H,z) dz. 
$

The next result is due to Helgason \cite{Helgason1,Helgason3}, Ludwig \cite{Helgason} and Solmon \cite{Solmon2}. 

\begin{lemma} \label{l:Sobmapping} The $(n-k)$-dimensional Radon transform is a bijection of $W_c^{s}(\R^n)$ onto $W_H^{s + \frac{n-k}{2}}(\mathcal{G}_{n-k,n})$. 
    
\end{lemma}

\section{A classification result for Problem~\ref{p:cpldRt}} \label{sec:solution}

In this section, we provide a classification for the comparison problem, Problem~\ref{p:NonhomogeneousBPproblem}. The case $p =1$ is relatively simple, while the case $p \neq 1$ requires a bit more work. We begin with the former case, $p =1$. 

\begin{proposition}\label{prop:p=1} Let $0 < k < n$ be an integer. Given a pair of continuous functions $g, h \colon \mathcal{G}_{n-k,n} \to [0,\infty)$ in $L^1(\mathcal{G}_{n-k,n},w)$ such that $\mathcal{R}_{n-k}^*g(x) \leq \mathcal{R}_{n-k}^*h(x)$ holds for all $x \in \R^n$, then 
\[
\|g\|_{L^1(\mathcal{G}_{n-k,n})} \leq\|h\|_{L^1(\mathcal{G}_{n-k,n})}.
\]
\end{proposition}

\begin{proof} To begin, let $a > 0$ and consider the function $f(x) = e^{-a |x|_2^2}$. We can to compute $\mathcal{R}_{n-k}f$ explicitly.  Fix an arbitrary $(H,z) \in \mathcal{G}_{n-k,n}$. Integrating in polar coordinates in $H$, followed by the change of variables $t = (\sqrt{a}r)^2$, we obtain 
\begin{align*}
&\mathcal{R}_{n-k}f(H,z) = \int_{H}f(y+z) dy\\
&= \int_{H \cap S^{n-1}} \int_0^\infty r^{n-k-1}f(r\theta +z) dr d\theta= \int_{H \cap S^{n-1}} \int_0^\infty r^{n-k-1} e^{-a[|z|_2^2 + r^2]} dr d\theta\\
&= |S^{n-k-1}| e^{-a |z|_2^2} \int_0^\infty  e^{-(\sqrt{a}r)^2} r^{n-k-1}dr= C(n,k,a) \left(\frac{\pi}{a}\right)^{\frac{n-k}{2}} e^{-a |z|_2^2},
\end{align*}
where in the third step we have used the fact that $y$ and $z$ are orthogonal, and where 
\[
C(n,k,a) =  |S^{n-k-1}| \Gamma\left(\frac{n-k}{2}\right) a^{-\frac{n-k}{2}}.
\]
Thus, 
\[
\mathcal{R}_{n-k}f(H,z) = C(n,k,a) e^{-a |z|_2^2} \quad \text{ for all } (H,z) \in \mathcal{G}_{n-k,n}. 
 \]

For each $\e >0$ consider the function
$
f_{\e}(x) = \frac{1}{C(n,k,\e)}e^{-\e |x|_2^2}.$ Fix $\e> 0$. Multiplying both sides of the inequality $\mathcal{R}^*g(x) \leq \mathcal{R}^*h(x)$ by $f_{\e}$, and using duality, we observe
\begin{align*}
&\int_{G_{n-k,n}}\int_{H^\perp} g(H,z) e^{-\e |z|_2^2}dz d\nu_{n-k}(H)= \int_{G_{n-k,n}}\int_{H^\perp} g(H,z) \mathcal{R}_{n-k}f_{\e}(H,z) dz d\nu_{n-k}(H)\\
&= \int_{\R^n} f_{\e}(x) \mathcal{R}_{n-k}^*g(x) dx \leq \int_{\R^n} f_{\e}(x) \mathcal{R}_{n-k}^*h(x) dx\\
&=\int_{G_{n-k,n}}\int_{H^\perp} h(H,z)\mathcal{R}_{n-k}f_{\e}(H,z)dz d\nu_{n-k}(H)\\
&=\int_{G_{n-k,n}}\int_{H^\perp} h(H,z) e^{-\e |z|_2^2}dz d\nu_{n-k}(H). 
\end{align*}
We have shown that 
\[ \int_{G_{n-k,n}}\int_{H^\perp} [h(H,z) -g(H,z)] e^{-\e |z|_2^2}dz d\nu_{n-k}(H) \geq 0 \] 
is valid for every $\e>0$. Since $h-g \in L^1(\mathcal{G}_{n-k,n})$, and since $e^{-\e |z|_2^2} \leq 1$ for every $\e>0$. Thus, we can apply the dominated convergence theorem to conclude that $\|g\|_{L^1(\mathcal{G}_{n-k,n})} \leq\|h\|_{L^1(\mathcal{G}_{n-k,n})}.$ 
\end{proof}

Denote by $M^+(\R^n)$ the subset of $M(\R^n)$ of all nonnegative measures. Recall the following definition.  

\begin{definition} Let $0 < k <n$ be an integer, and let $p >1$. We define the class of $(p,k)$-admissible functions by 
\[
\mathcal{A}_p^{k} := \left\{h \in C_0(\mathcal{G}_{n-k,n}) \colon h \geq 0 \text{ and } h^{p-1} = \mathcal{R}_{n-k} \mu \text{ for some } \mu \in \mathcal{M}^+(\R^n) \right\}. 
\]
We also have the following important subclass of $\mathcal{I}_p^k$:
\begin{align*}
\mathcal{A}_{p,\infty}^k:= &\left\{h \in C_H^\infty(\mathcal{G}_{n-k,n}) \colon h \geq 0, h^{p-1} \in C_H^\infty(\mathcal{G}_{n-k,n}), \right.\\
&\left. \text{ and } h^{p-1} = \mathcal{R}_{n-k} \varphi  \text{ for some } \varphi \in C_c^\infty(\R^n), \varphi \geq 0 \right\}.
\end{align*}
\end{definition}

It should be noted that (by the proof of Proposition~\ref{prop:p=1}) functions $h\colon \mathcal{G}_{n-k,n} \to \R$ of the form $h(H,z) = e^{-a|z|_2^2}, a >0$ belong to the class $\mathcal{A}_{p,k}$ for every $p > 1$. 

With the class of $(p,k)$-admissible functions in hand, we are in a position to give an affirmative answer to Problem~\ref{p:cpldRt}.

\begin{theorem}\label{t:affirmative} Let $0< k < n$ be an integer, and let $p>1$. Let $g,h \in C_0(\mathcal{G}_{n-k,n}) \cap L^p(\mathcal{G}_{n-k,n})$ be such that $\mathcal{R}^*g(x) \leq \mathcal{R}^*h$ holds. If $g \in \mathcal{A}_p^k$, then \[\|g\|_{L^p(\mathcal{G}_{n-k,n})}\leq \|h\|_{L^p(\mathcal{G}_{n-k,n})}.\]
  
\end{theorem}

\begin{proof}   As $g \in \mathcal{A}_p^k$, there is a measure $\mu_g \in \mathcal{M}^+(\R^n)$ such that $g^{p-1} = \mathcal{R}_{n-k} \mu_g$. Using the definition of $\mathcal{R}\mu_g$ together with the assumption $\mathcal{R}^*g(x) \leq \mathcal{R}^*h(x)$, for all $x \in \R^n$, we may use the duality relations followed by H\"older's inequality to write
\begin{align*}
\|g\|_{L^p(\mathcal{G}_{n-k,n},w)}^p &= \int_{G_{n-k,n}}\int_{H^\perp} g^p(H,z)dz d\nu_{n-k}(H)\\
&= \langle g^{p-1}, g \rangle_k = \langle \mathcal{R}_{n-k}\mu_g, g\rangle_k= \langle \mu_g, \mathcal{R}_{n-k}^*g \rangle\\
&= \int_{\R^n} \mathcal{R}_{n-k}^*g(x) d\mu_g(x) \leq \int_{\R^n} \mathcal{R}_{n-k}^*h(x) d\mu_g(x)\\
&= \langle \mu_g, \mathcal{R}_{n-k}^*h \rangle = \langle \mathcal{R}_{n-k}\mu_g, h\rangle_k=  \langle g^{p-1}, h \rangle_k\\
&= \int_{G_{n-k,n}}\int_{H^\perp}g^{p-1}(H,z) h(H,z) dz d\nu_{n-k}(H)\\
&\leq \|g\|_{L^p(\mathcal{G}_{n-k,n})}^{p-1} \|h\|_{L^p(\mathcal{G}_{n-k,n})}.
\end{align*}
This proves the desired result. 
\end{proof}

Our next result shows that the class of $(k,p)$-admissible functions actually classifies the affirmative answer to Problem~\ref{p:cpldRt}.  

\begin{theorem}\label{t:counterexample} Let $0< k <n$ be an integer, and let $p >1$. Let $h \geq 0$ belong to $C_H^\infty(\mathcal{G}_{n-k,n})\setminus \mathcal{A}_{p,\infty}^k$. Then there exists some non-negative function $g \in C_c^\infty(\mathcal{G}_{n-k,n})$ such that $\mathcal{R}_{n-k}^*g  \leq\mathcal{R}_{n-k}^*h $, and yet $\|g\|_{L^p(\mathcal{G}_{n-k,n})} > \|h\|_{L^p(\mathcal{G}_{n-k,n})}$. 
\end{theorem}

\begin{proof} Consider some $h\geq 0$ belonging to  $C_H^\infty(\mathcal{G}_{n-k,n}) \setminus \mathcal{A}_{p,\infty}^k$. According to Lemma~\ref{l:additionalmapping}, there is a $\varphi \in C_c^{\infty}(\R^n)$ such that $h^{p-1} = \mathcal{R}_{n-k}\varphi$. Now, as $h \not \in \mathcal{A}_{p,\infty}^k$, there must be a non-empty bounded open set $\mathcal{O} \subset \R^n$ on which $\varphi$ is negative. In view of Corollary~\ref{coro:mapping}, select a function $\eta \in \mathcal{R}_{n-k}^*(C^\infty_c(\mathcal{G}_{n-k,n}))$ such that $\eta > 0$ on $ \mathcal{O} \cup (\R^n \setminus \text{supp}(\varphi))$ and $\eta = 0$ on $\text{supp}(\varphi) \setminus \mathcal{O}$, and again using Corollary~\ref{coro:mapping}, select $f \in C_c^\infty(\mathcal{G}_{n-k,n})$ such that $\eta = \mathcal{R}_{n-k}^* f$.

Define $g \in C_c^\infty(\mathcal{G}_{n-k,n})$ by $g = h- \e f$, where $\e >0$ is chosen small enough so that $g \geq 0$. 
Observe that, for every $x \in \R^n$, we have 
$
\mathcal{R}_{n-k}^*g = \mathcal{R}_{n-k}^*[h-\e f]= \mathcal{R}_{n-k}^*h- \e \eta \leq \mathcal{R}_{n-k}^*h$.
Next, invoking H\"older's inequality, we obtain 
\begin{align*}
&\|h\|_{L^p(\mathcal{G}_{n-k,n})}^p = \int_{G_{n-k,n}} \int_{H^\perp} h^p(H,z) dz d\nu_{n-k}(H)\\
&= \langle h^{p-1},h\rangle_k= \langle \mathcal{R}_{n-k} \varphi, h\rangle_k= \langle \mathcal{R}_{n-k}^*h, \varphi \rangle\\
&= \int_{\R^n}\varphi(x) \mathcal{R}_{n-k}^*h(x) dx\\
&< \int_{\R^n} \varphi(x) [\mathcal{R}_{n-k}^*h(x)-\e\eta(x)]dx =\int_{\R^n} \varphi(x) \mathcal{R}_{n-k}^*[h-\e f](x)dx\\
&= \langle R_{n-k}^*g, \varphi \rangle = \langle R_{n-k} \varphi, g \rangle_k =\langle h^{p-1},g\rangle_k\\
&=\int_{G_{n-k,n}} \int_{H^\perp} g(H, z) h^{p-1}(H,z) dz d\nu_{n-k}(H)\\
&\leq \|g\|_{L^p(\mathcal{G}_{n-k,n})} \|h\|_{L^p(\mathcal{G}_{n-k,n})}^{p-1}
\end{align*}
as required. This completes the proof. 
\end{proof}

Note that Theorem~\ref{t:basiccase} can be immediately deduced from Proposition~\ref{prop:p=1}, Theorem~\ref{t:affirmative} and Theorem~\ref{t:counterexample}.

\section{A slicing inequality for dual Radon transforms} \label{sec:slicing}

Inspired by the methods used to solve Problem~\ref{p:cpldRt}, we see to establish a general slicing inequality for dual Radon transforms. We begin with a definition. 

\begin{definition}\label{d:distance} Let $0 < k < n$ be an integer, $p >1$, and $w \colon \mathcal{G}_{n-k,n} \to [0,\infty)$ be a bounded continuous function.  Given a non-negative function $g \in C_0(\mathcal{G}_{n-k,n}) \cap L^p(\mathcal{G}_{n-k,n,w})$, we define the $(L^p,w)$-distance from $g$ to the class of $(p,k)$-admissible functions, $\mathcal{A}_{p,k}$, by 
\[
d_{p,w}(g,\mathcal{A}_p^k) := \inf \left\{\frac{\|f\|_{L^p(\mathcal{G}_{n-k,n},w)}}{\|g\|_{L^p(\mathcal{G}_{n-k,n},w)}} \colon f \geq 0, f \in \mathcal{A}_p^k \cap L^p(\mathcal{G}_{n-k,n},w) \right\}. 
\]
If $w \equiv 1$, then we simply write $d_{p}(g,\mathcal{A}_p^k)$.
\end{definition}

The main result in this section is the following analogue of the inequalities \eqref{e:functionalBourgain} and \eqref{e:KRZslicing} for the dual $(n-k)$-dimensional Radon transform.

\begin{theorem}\label{t:generalslicing} Let $0 < k < n$ be an integer, $p > 1$, and $h \geq 0$ and  $w \geq 0$ be a pair of functions in $C_b(\mathcal{G}_{n-k,n})$ such that $hw \in C_0(\mathcal{G}_{n-k,n})$, $h \in L^p(\mathcal{G}_{n-k,n},w)$ and $w \in  L^1(\mathcal{G}_{n-k,n})$. For any non-negative function $g \in C_0(\mathcal{G}_{n-k,n}) \cap L^p(\mathcal{G}_{n-k,n},w)$, the next estimate is valid: 
\begin{equation}\label{e:slicinginequality}
\frac{\|g\|_{L^p(\mathcal{G}_{n-k,n},w)}}{\|h\|_{L^p(\mathcal{G}_{n-k,n},w)}} \leq d_{p,w}(g,\mathcal{A}_p^k)^{p-1} \left\|\frac{\mathcal{R}_{n-k}^*[gw]}{\mathcal{R}_{n-k}^*[hw]}\right\|_{L^\infty(\R^n)}.
\end{equation}
\end{theorem}

Theorem~\ref{t:basicslicing} can be deduced from the above theorem by applying \cite[Theorem~7.7(a)]{Solmon}. 

\begin{proof} Let $\e >0$ be such that $\mathcal{R}_{n-k}^*[gw](x) \leq \e \mathcal{R}_{n-k}^*[hw](x)$ holds for every $x \in \R^n$, and consider $f \in  \mathcal{A}_p^k \cap L^p(\mathcal{G}_{n-k,n},w)$ such that 
\[
f \geq g \quad \text{ and } \|f\|_{L^p(\mathcal{G}_{n-k,n},w)} \leq (1+\delta) d_{p,w}(g,\mathcal{A}_p^k) \|g\|_{L^p(\mathcal{G}_{n-k,n},w)}.
\]
As $f \in \mathcal{A}_p^k$, there must be a measure $\mu \in M^+(\R^n)$ such that $f^{p-1} = \mathcal{R}_{n-k} \mu$ as functionals on $C_0(\mathcal{G}_{n-k,n})$. Using the definition of $f$, the fact that $gw, hw \in C_0(\mathcal{G}_{n-k,n})$, followed by H\"older's inequality, and then again the choice of $f$, we can write 
\begin{align*}
\|g\|_{L^p(\mathcal{G}_{n-k,n},w)}^p 
&= \int_{G_{n-k,n}} g^{p-1}(H,z) g(H,z) w(H,z) dz d\nu_{n-k}(H)\\
&\leq \int_{G_{n-k,n}} f^{p-1}(H,z) g(H,z) w(H,z) dz d\nu_{n-k}(H)\\
&= \langle f^{p-1}, gw \rangle_k= \langle \mu, \mathcal{R}_{n-k}^*[gw] \rangle\\
&= \int_{\R^n} \mathcal{R}_{n-k}^*[gw](x) d\mu(x)\leq \e\int_{\R^n} \mathcal{R}_{n-k}^*[hw](x) d\mu(x)\\
&= \e \langle \mu, \mathcal{R}_{n-k}^*[hw] \rangle = \e \langle f^{p-1}, hw \rangle_k\\
&= \e \int_{G_{n-k,n}} f^{p-1}(H,z) h(H,z) w(H,z) dz d\nu_{n-k}(H)\\
&\leq \e \|f\|_{L^p(\mathcal{G}_{n-k,n},w)}^{p-1} \|h\|_{L^p(\mathcal{G}_{n-k,n},w)}\\
&\leq \e (1+\delta)^{p-1}d_{p,w}(g,\mathcal{A}_p^k)^{p-1}]\|g\|_{L^p(\mathcal{G}_{n-k,n},w)}^{p-1}\|h\|_{L^p(\mathcal{G}_{n-k,n},w)}.
\end{align*}
Select
\[
\e = \sup_{x \in \R^n}\left[\frac{\mathcal{R}_{n-k}^*[gw](x)}{\mathcal{R}_{n-k}^*[hw](x)} \right].
\]
Finally, sending $\delta \to 0^+$, we obtain the inequality \eqref{e:slicinginequality}. 

\end{proof}

We now prove Theorem~\ref{theorem:slicingaverage}, Theorem~\ref{t:slicing} and Theorem~\ref{t:meanvalue}.

\begin{proof}[Proof of Theorem~\ref{theorem:slicingaverage}]
By assumption, we have 
$
\frac{\beta}{\gamma}f\leq g \leq f$,
with $f(H,z) = \gamma e^{-\alpha |z|_2^2}$. We also have that $f^{p-1}$ belongs to $\mathcal{A}_{p,k}$. 
Raising both sides of the inequality $\frac{\beta}{\gamma} f(H,z) \leq g(H,z)$ to the $p$-th power, multiplying by $w(H,z)$ and integrating first over $H^\perp$ and then over $G_{n-k,n}$, and finally taking the $p$-th root, we obtain 
$
\frac{\beta}{\gamma} \|f\|_{L^p(\mathcal{G}_{n-k,n},w)} \leq \|g\|_{L^p(\mathcal{G}_{n-k,n},w)},$
which implies that $d_{p,w}(g,\mathcal{A}_{p,k}) \leq \frac{\gamma}{\beta}$,
and the theorem follows from Theorem~\ref{t:generalslicing} with $g = g$, $w=w$, and $h =1$.

\end{proof}

\begin{proof}[Proof of Theorem~\ref{t:slicing}] 

Suppose that $g^{p-1} \in W_H^{s+ \frac{n-k}{2}}(\mathcal{G}_{n-k,n})$ is non-negative and continuous. According to Lemma~\ref{l:Sobmapping}, there must exist some $\varphi \in W_c^s(\R^n)$ such that $g^{p-1} = \mathcal{R}_{n-k} \varphi$. According to Theorem~\ref{t:mappingproperty}, we then have that $\mathcal{R}_{n-k}^*g^{p-1} = O(|x|^{-k})$ as $|x|_2 \to \infty$.  Similarly, we note that $g = [\mathcal{R}_{n-k} \varphi]^{\frac{1}{p-1}}$ is in $C_c(\mathcal{G}_{n-k,n})$, and again use Theorem~\ref{t:mappingproperty} to conclude that $\mathcal{R}_{n-k}^*g(x) = O(|x|_2^{-k})$ as $|x|_2 \to \infty$. Therefore, we have 
$\mathcal{R}_{n-k}^*g(x) [\mathcal{R}_{n-k}^*g^{p-1}(x)]^{-1} = O(1)$
as $|x|_2 \to \infty$. 
Since $g$ and $g^{p-1}$ share the same support, $\mathcal{R}_{n-k}^*g(x) [\mathcal{R}_{n-k}^*g^{p-1}(x)]^{-1}$ is bounded. Finally, applying the inequality \eqref{e:slicinginequality} with $g = g$, $h = g^{p-1}$ and $w =1$, we obtain the desired result. 
\end{proof}

\begin{proof}[Proof of Theorem~\ref{t:meanvalue}] Since $g^{p-1} \in W_H^{s+\frac{n-k}{2}}(\mathcal{G}_{n-k,n})$, according to Lemma~\ref{l:Sobmapping}, there exists a function $\varphi \in W_c^s(\R^n)$ such that $g^{p-1} = \mathcal R_{n-k}\varphi$, and necessarily $g=[\mathcal{R}_{n-k}\varphi]^{\frac{1}{p-1}}$ is in $C_c(\mathcal{G}_{n-k,n})$, and so by employing Theorem~\ref{t:mappingproperty} it then follows that $\mathcal{R}_{n-k}^*[gw] = O(|x|_2^{-k})$ as $|x|_2 \to \infty$. Similarly, $\mathcal{R}_{n-k}^*w(x) = O(|x|_2^{-k})$ as $|x|_2 \to \infty$. Hence, 
\[
\mathcal{R}_{n-k}^*[gw](x) [\mathcal{R}_{n-k}^*w(x)]^{-1} = O(1) \text{ as } |x|_2 \to \infty. 
\]
Since $\text{supp}(gw) \subset \text{supp}(w)$, the result follows from Theorem~\ref{t:generalslicing} with $g = g$, $h=1$ and $w=w$.  
    
\end{proof}


{\bf Acknowledgments:} The author would like to thank Semyon Alesker, Dimitry Ryabogin and Artem Zvavitch for fruitful discussions. This material is based upon work supported by the National Science Foundation under
Grant No. DMS-1929284 while the authors were in residence at the Institute for Computational
and Experimental Research in Mathematics in Providence, RI, during the Harmonic Analysis and
Convexity program December 2024.









\end{document}